\documentclass[a4paper,12pt]{article}

\usepackage{abstract}
\usepackage{parskip}

\usepackage{titlesec}
\def\sectionfont{\sffamily\Large\bfseries\boldmath}
\def\subsectionfont{\sffamily\large\bfseries\boldmath}
\def\paragraphfont{\sffamily\normalsize\bfseries\boldmath}
\titleformat*{\section}{\sectionfont}
\titleformat*{\subsection}{\subsectionfont}
\titleformat*{\subsubsection}{\paragraphfont}
\titleformat*{\paragraph}{\paragraphfont}
\titleformat*{\subparagraph}{\paragraphfont}
\usepackage[small,labelfont={bf,sf}]{caption}
\usepackage{fullpage,graphicx,psfrag,url}
\usepackage{multirow}
\usepackage{enumitem}
\setlist{nolistsep}

\usepackage{booktabs}
\usepackage{adjustbox}

\usepackage{amsmath,amsthm}

\newtheoremstyle{exampstyle}
  {.5\baselineskip}
  {\topsep}
  {}
  {}
  {\bfseries}
  {.}
  {.5em}
  {}
\theoremstyle{exampstyle}

\newtheorem{corollary}{Corollary}[section]
\newtheorem{fact}[corollary]{Fact}

\newtheorem{proposition}[corollary]{Proposition}
\newtheorem{remark}[corollary]{Remark}

\usepackage{cite}
\usepackage[english]{babel}
\usepackage{amssymb,mathtools,amsfonts}
\usepackage[mathscr]{eucal}
\usepackage{enumitem}
\usepackage{hyperref}
\hypersetup{%
  colorlinks=true,
  linktoc=all,
  linkcolor=black,
  citecolor= black,
  urlcolor=blue
}
\newcommand{\eg}{\emph{e.g.},\ }
\newcommand{\ie}{\emph{i.e.},\ }

\newcommand{\mcf}{\mathcal}
\newcommand{\mbb}{\mathbb}

\newcommand{\Nat}{\mbb{N}}

\newcommand{\eqdef}{\coloneqq}

\newcommand{\Prop}{Prop.}

\newcommand{\Cor}{Cor.}
\newcommand{\Thm}{Thm.}
\newcommand{\Lem}{Lem.}
\newcommand{\Ex}{Example}

\DeclareMathOperator*{\argmin}{\operatorname{argmin}}
\DeclareMathOperator{\Id}{Id}
\DeclareMathOperator{\gra}{gra}
\DeclareMathOperator{\prox}{Prox}
\DeclareMathOperator{\range}{ran}

\newcommand{\indicator}[1]{\iota_{#1}}
\newcommand{\project}[1]{P_{#1}}
\newcommand{\normalCone}[1]{N_{#1}}
\newcommand{\polar}[1]{{#1}^\ominus}
\newcommand{\recession}[1]{{\rm rec}\,{#1}}
\newcommand{\support}[1]{\sigma_{#1}}
\newcommand{\closure}[1]{\overline{#1}}
\newcommand{\inv}{^{-1}}
\newcommand{\norm}[1]{\lVert#1\rVert}
\newcommand{\half}{\tfrac{1}{2}}
\newcommand{\innerprod}[2]{\left\langle{#1}\mid{#2}\right\rangle}
\newcommand{\seq}[1]{({#1})_{n\in\Nat}}
\newcommand{\dlt}{\delta}
\newcommand{\setC}{C}
\newcommand{\setD}{D}
\newcommand{\stepsize}{\gamma}

\newcommand{\OptProblem}[5][]{
  \begin{aligned}\label{#1}
    \begin{array}{ll}
    \underset{#3}{\rm{#2}}&\hspace{-0ex}#4\vspace{.5ex}\\
    \rm{subject~to}&#5
  \end{array}
\end{aligned}
}
\newcommand{\MinProblem}[4][]
{\OptProblem[#1]{minimize}{#2}{#3}{#4}}

\title{\bfseries\sffamily On the Asymptotic Behavior of the Douglas-Rachford and Proximal-Point Algorithms for Convex Optimization}
\author{Goran Banjac and John Lygeros}

\begin{document}
\maketitle

\begin{abstract}
The authors in \cite{Banjac:2019} recently showed that the Douglas-Rachford algorithm provides certificates of infeasibility for a class of convex optimization problems.
In particular, they showed that the difference between consecutive iterates generated by the algorithm converges to certificates of primal and dual strong infeasibility.
Their result was shown in a finite dimensional Euclidean setting and for a particular structure of the constraint set.
In this paper, we extend the result to real Hilbert spaces and a general nonempty closed convex set.
Moreover, we show that the proximal-point algorithm applied to the set of optimality conditions of the problem generates similar infeasibility certificates.
\end{abstract}

\section{Introduction}\label{sec:intro}
Due to its very good practical performance and ability to handle nonsmooth functions, the Douglas-Rachford algorithm has attracted a lot of interest for solving convex optimization problems.
Provided that a problem is solvable and satisfies certain constraint qualification, the algorithm converges to an optimal solution \cite[\Cor~27.3]{Bauschke:2017:book}.
If the problem is infeasible, then some of its iterates diverge \cite{Eckstein:1992}.

Results on the asymptotic behavior of the Douglas-Rachford algorithm for infeasible problems are very scarce, and most of them study some specific cases such as feasibility problems involving two convex sets that do not intersect \cite{Bauschke:2016a,Bauschke:2016b,Bauschke:2017}.
Although there have been some recent results studying a more general setting \cite{Ryu:2019,Bauschke:2020}, they impose some additional assumptions on feasibility of either the primal or the dual problem.
The authors in \cite{Banjac:2019} consider a problem of minimizing a convex quadratic function over a particular constraint set, and show that the iterates of the Douglas-Rachford algorithm generate an infeasibility certificate when the problem is primal and/or dual strongly infeasible.
A similar analysis was applied in \cite{Liao-McPherson:2020} to show that the proximal-point algorithm used for solving a convex quadratic program can also detect infeasibility.

The constraint set of the problem studied in \cite{Banjac:2019} is represented in the form $Ax\in\setC$, where $A$ is a real matrix and $\setC$ the Cartesian product of a convex compact set and a translated closed convex cone.
This paper extends the result of \cite{Banjac:2019} to real Hilbert spaces and a general nonempty closed convex set $\setC$.
Moreover, we show that a similar analysis can be used to prove that the proximal-point algorithm for solving the same class of problems generates similar infeasibility certificates.

The paper is organized as follows.
We introduce some definitions and notation in the sequel of Section~\ref{sec:intro}, and the problem under consideration in Section~\ref{sec:problem}.
Section~\ref{sec:aux} presents some supporting results that are essential for generalizing the results in \cite{Banjac:2019}.
Finally, Section~\ref{sec:dra} and Section~\ref{sec:ppa} analyze the asymptotic behavior of the Douglas-Rachford and proximal-point algorithms, respectively, and show that they provide infeasibility certificates for the considered problem.

\subsection{Notation}
Let $\mcf{H}$, $\mcf{H}_1$, $\mcf{H}_2$ be real Hilbert spaces with inner products $\innerprod{\cdot}{\cdot}$, induced norms $\norm{\,\cdot\,}$, and identity operators $\Id$.
The power set of $\mcf{H}$ is denoted by $2^\mcf{H}$.
Let $\Nat$ denote the set of positive integers.
For a sequence $\seq{s_n}$, we denote by $s_n\to s$ ($s_n\rightharpoonup s$) that it converges strongly (weakly) to $s$ and define $\dlt s_{n+1}\eqdef s_{n+1}-s_n$.

Let $\setD$ be a nonempty subset of $\mcf{H}$ with $\closure{\setD}$ being its \emph{closure}.
Then $T\colon\setD\to\mcf{H}$ is \emph{nonexpansive} if
\[
  (\forall x\in\setD)(\forall y\in\setD) \quad \norm{Tx-Ty} \le \norm{x-y},
\]
and it is \emph{$\alpha$-averaged} with $\alpha\in\left]0,1\right[$ if there exists a nonexpansive operator $R\colon\setD\to\mcf{H}$ such that $T=(1-\alpha)\Id + \alpha R$.
We denote the range of $T$ by $\range T$.
A set-valued operator $B\colon\mcf{H}\to2^\mcf{H}$, characterized by its \emph{graph}
\[
  \gra B = \left\lbrace (x,u)\in\mcf{H}\times\mcf{H} \mid u\in Bx \right\rbrace,
\]
is \emph{monotone} if
\[
  \left( \forall(x,u)\in\gra B \right)\left( \forall(y,v)\in\gra B \right) \quad \innerprod{x-y}{u-v} \ge 0.
\]
The \emph{inverse} of $B$, denoted by $B\inv$, is defined through its graph
\[
  \gra B\inv = \left\lbrace (u,x)\in\mcf{H}\times\mcf{H} \mid (x,u)\in\gra B \right\rbrace.
\]
For a proper lower semicontinuous convex function $f\colon\mcf{H}\to \left]-\infty,+\infty\right]$, we define its:
\begin{align*}
  &\text{\emph{Fenchel conjugate}:} && f^* \colon \mcf{H}\to\left]-\infty,+\infty\right] \colon u\mapsto\sup_{x\in\mcf{H}}\left( \innerprod{x}{u} - f(x) \right), \\
  &\text{\emph{proximity operator}:} && \prox_f \colon \mcf{H}\to\mcf{H} \colon x\mapsto\argmin_{y\in\mcf{H}}\left( f(y) + \half \norm{y-x}^2 \right), \\
  &\text{\emph{subdifferential}:} && \partial f \colon \mcf{H}\to2^\mcf{H} \colon x\mapsto\left\lbrace u\in\mcf{H} \mid (\forall y\in\mcf{H}) \: \innerprod{y-x}{u}+f(x)\le f(y) \right\rbrace.
\end{align*}
For a nonempty closed convex set $\setC\subseteq\mcf{H}$, we define its:
\allowdisplaybreaks
\begin{align*}
  &\text{\emph{polar cone}:} && \polar{\setC} = \Big\lbrace u\in\mcf{H} \mid \sup_{x\in\setC}\innerprod{x}{u} \le 0 \Big\rbrace, \\
  &\text{\emph{recession cone}:} && \recession{\setC} = \left\lbrace x\in\mcf{H} \mid (\forall y\in\setC) \: x+y\in\setC \right\rbrace, \\
  &\text{\emph{indicator function}:} && \indicator{\setC} \colon \mcf{H}\to\left[0,+\infty\right] \colon x\mapsto\begin{cases} 0 & x\in\setC \\ +\infty & \text{otherwise,}\end{cases} \\
  &\text{\emph{support function}:} && \support{\setC} \colon \mcf{H}\to\left]-\infty,+\infty\right] \colon u\mapsto\sup_{x\in\setC} \innerprod{x}{u}, \\
  &\text{\emph{projection operator}:} && \project{\setC} \colon \mcf{H}\to\mcf{H} \colon x\mapsto\argmin_{y\in\setC}\,\norm{y-x}, \\
  &\text{\emph{normal cone operator}:} && \normalCone{\setC} \colon \mcf{H}\to2^\mcf{H} \colon x\mapsto \begin{cases} \big\lbrace u\in\mcf{H} \mid \sup_{y\in\setC} \innerprod{y-x}{u}\le 0 \big\rbrace & x\in\setC \\ \emptyset & x\notin\setC. \end{cases}
\end{align*}

\section{Problem of Interest}\label{sec:problem}
Consider the following convex optimization problem:
\begin{equation}
	\MinProblem[{eqn:primal}]{x\in\mcf{H}_1}{\half \innerprod{Qx}{x} + \innerprod{q}{x}}{Ax\in\setC,}
\end{equation}
with $Q\colon\mcf{H}_1\to\mcf{H}_1$ a monotone self-adjoint bounded linear operator, $q\in\mcf{H}_1$, $A\colon\mcf{H}_1\to\mcf{H}_2$ a bounded linear operator, and $\setC$ a nonempty closed convex subset of $\mcf{H}_2$; we assume that $\range{Q}$ and $\range{A}$ are closed.
The objective function of the problem is convex, continuous, and Fr\'echet differentiable \cite[\Prop~17.36]{Bauschke:2017:book}.

When $\mcf{H}_1$ and $\mcf{H}_2$ are finite-dimensional Euclidean spaces, problem~\eqref{eqn:primal} reduces to the one considered in \cite{Banjac:2019}, where the Douglas-Rachford algorithm (which is equivalent to the alternating direction method of multipliers) was shown to generate certificates of primal and dual strong infeasibility.
Moreover, the authors proposed termination criteria for infeasibility detection, which are easy to implement and are used in several numerical solvers; see, \eg \cite{Stellato:2020,Garstka:2019,Hermans:2019}.
To prove the main results, they used the assumption that $\setC$ can be represented as the Cartesian product of a convex compact set and a translated closed convex cone, which was exploited heavily in their proofs.
In this paper we extend these results to the case where $\mcf{H}_1$ and $\mcf{H}_2$ are real Hilbert spaces, and $\setC$ is a general nonempty closed convex set.

\subsection{Optimality Conditions}
We can rewrite problem~\eqref{eqn:primal} in the form
\[
  \underset{x\in\mcf{H}_1}{\rm minimize} \quad \half \innerprod{Qx}{x} + \innerprod{q}{x} + \indicator{\setC}(Ax).
\]
Provided that a certain constraint qualification holds, we can characterize its solution by \cite[\Thm~27.2]{Bauschke:2017:book}
\[
  0 \in Qx + q + A^* \partial \indicator{\setC}(Ax),
\]
and introducing a dual variable $y\in\partial\indicator{\setC}(Ax)$, we can rewrite the inclusion as
\begin{equation}\label{eqn:opt_incl}
  0 \in \begin{pmatrix} Qx + q + A^* y \\ -y + \partial \indicator{\setC}(Ax) \end{pmatrix}.
\end{equation}
Introducing an auxiliary variable $z\in\setC$ and using $\partial\indicator{\setC}=\normalCone{\setC}$, we can write the optimality conditions for problem~\eqref{eqn:primal} as
\begin{subequations}\label{eqn:kkt}
\begin{align}
  Ax-z=0 \label{eqn:prim_feas} \\
  Qx + q + A^* y = 0 \label{eqn:dual_feas} \\
  z\in\setC, \quad y\in\normalCone{\setC} z. \label{eqn:inclusion_cond}
\end{align}
\end{subequations}

\subsection{Infeasibility Certificates}
The authors in \cite{Banjac:2019} derived the following conditions for characterizing strong infeasibility of problem~\eqref{eqn:primal} and its dual:
\begin{proposition}[\hspace{1sp}{\cite[\Prop~3.1]{Banjac:2019}}]\label{prop:infeas}~
\begin{enumerate}[label=(\roman*)]
\item If there exists a $\bar{y}\in\mcf{H}_2$ such that
\[
  A^* \bar{y} = 0
  \quad\text{and}\quad
  \support{\setC}(\bar{y}) < 0,
\]
then problem~\eqref{eqn:primal} is strongly infeasible.
\item If there exists an $\bar{x}\in\mcf{H}_1$ such that
\[
  Q\bar{x}=0,
  \quad
  A\bar{x}\in\recession{\setC},
  \quad\text{and}\quad
  \innerprod{q}{\bar{x}} < 0,
\]
then the dual of problem~\eqref{eqn:primal} is strongly infeasible.
\end{enumerate}
\end{proposition}

\section{Auxiliary Results}\label{sec:aux}

\begin{fact}\label{fact:operatorConvergence}
Suppose that $T\colon\mcf{H}\to\mcf{H}$ is an averaged operator and let $s_0\in\mcf{H}$, $s_n=T^n s_0$, and $\dlt s \eqdef \project{\closure{\range}(T-\Id)}(0)$.
Then
\begin{enumerate}[label=(\roman*)]
  \item \label{lem:operatorConvergence:i} $\tfrac{1}{n} s_n \to \dlt s$.
  \item \label{lem:operatorConvergence:ii} $\dlt s_n \to \dlt s$.
\end{enumerate}
\end{fact}
\begin{proof}
The first result is \cite[\Cor~3]{Pazy:1971} and the second is \cite[\Cor~2.3]{Baillon:1978}.
\end{proof}

The following proposition provides essential ingredients for generalizing the results in \cite[\S 5]{Banjac:2019}.

\begin{proposition}\label{prop:dpdr}
Let $\seq{s_n}$ be a sequence in $\mcf{H}$ satisfying $\tfrac{1}{n}s_n\to\dlt s$.
Let $\setD\subseteq\mcf{H}$ be a nonempty closed convex set and define sequences $\seq{p_n}$ and $\seq{r_n}$ by
\begin{align*}
  p_n &\eqdef \project{\setD} s_n \\
  r_n &\eqdef (\Id - \project{\setD}) s_n.
\end{align*}
Then
\begin{enumerate}[label=(\roman*)]
  \item \label{prop:dpdr:i}   $r_n \in \polar{(\recession{\setD})}$.
  \item \label{prop:dpdr:ii}  $\tfrac{1}{n} p_n \to \dlt p \eqdef \project{\recession{\setD}}(\dlt s)$.
  \item \label{prop:dpdr:iii} $\tfrac{1}{n} r_n \to \dlt r \eqdef \project{\polar{(\recession{\setD})}}(\dlt s)$.
  \item \label{prop:dpdr:iv}  $\lim_{n\to\infty}\tfrac{1}{n}\innerprod{p_n}{r_n} = \support{\setD}(\dlt r)$.
\end{enumerate}
\end{proposition}
\begin{proof}
\ref{prop:dpdr:i}:
Follows from \cite[\Thm~3.1]{Zarantonello:1971}.

\ref{prop:dpdr:ii}\&\ref{prop:dpdr:iii}:
A related result was shown in \cite[\Lem~6.3.13]{Facchinei:2003} and \cite[\Prop~2.2]{Gowda:2019} in a finite-dimensional setting.
Using similar arguments here, together with those in \cite[\Lem~4.3]{Shen:2015}, we can only establish the weak convergence, \ie $\tfrac{1}{n}p_n\rightharpoonup\dlt p$.
Using Moreau's decomposition \cite[\Thm~6.30]{Bauschke:2017:book}, it follows that $\tfrac{1}{n}r_n\rightharpoonup\dlt r$ and $\norm{\dlt s}^2=\norm{\dlt p}^2+\norm{\dlt r}^2$.
For an arbitrary vector $z\in\setD$, \cite[\Thm~3.16]{Bauschke:2017:book} yields
\[
  \norm{s_n-z}^2 \ge \norm{p_n-z}^2 + \norm{r_n}^2, \quad \forall n\in\Nat.
\]
Dividing the inequality by $n^2$ and taking the limit superior, we get
\[
  \lim \, \norm{\tfrac{1}{n}s_n}^2 \ge \overline{\lim} \, (\norm{\tfrac{1}{n}p_n}^2 + \norm{\tfrac{1}{n}r_n}^2) \ge \overline{\lim} \, \norm{\tfrac{1}{n}p_n}^2 + \underline{\lim} \, \norm{\tfrac{1}{n}r_n}^2,
\]
and thus
\[
  \overline{\lim} \, \norm{\tfrac{1}{n}p_n}^2 \le \lim \, \norm{\tfrac{1}{n}s_n}^2 - \underline{\lim} \, \norm{\tfrac{1}{n}r_n}^2 \le \norm{\dlt s}^2 - \norm{\dlt r}^2 = \norm{\dlt p}^2,
\]
where the second inequality follows from \cite[\Lem~2.42]{Bauschke:2017:book}.
The inequality above yields $\overline{\lim} \, \norm{\tfrac{1}{n}p_n} \le \norm{\dlt p}$, which due to \cite[\Lem~2.51]{Bauschke:2017:book} implies $\tfrac{1}{n}p_n\to \dlt p$.
Using Moreau's decomposition, it follows that $\tfrac{1}{n}r_n \to \dlt r$.

\ref{prop:dpdr:iv}:
Taking the limit of the inequality
\[
  (\forall n\in\Nat)(\forall\hat{p}\in\setD) \quad \innerprod{\hat{p}}{\tfrac{1}{n} r_n} \le \sup_{p\in\setD} \innerprod{p}{\tfrac{1}{n}r_n},
\]
we obtain
\[
  (\forall\hat{p}\in\setD) \quad \lim_{n\to\infty}\innerprod{\hat{p}}{\tfrac{1}{n}r_n} \le \lim_{n\to\infty}\sup_{p\in\setD} \innerprod{p}{\tfrac{1}{n}r_n},
\]
and taking the supremum of the left-hand side over $\setD$, we get
\begin{equation}\label{eqn:sup_lim_ineq}
  \sup_{p\in\setD}\lim_{n\to\infty}\innerprod{p}{\tfrac{1}{n}r_n} \le \lim_{n\to\infty}\sup_{p\in\setD} \innerprod{p}{\tfrac{1}{n}r_n}.
\end{equation}
From~\cite[\Prop~6.47]{Bauschke:2017:book}, we have
\[
  r_n = s_n - p_n \in \normalCone{\setD} p_n,
\]
which, due to \cite[\Thm~16.29]{Bauschke:2017:book} and the facts that $\indicator{\setD}^*=\support{\setD}$ and $\partial\indicator{\setD}=\normalCone{\setD}$, is equivalent to
\begin{equation}\label{eqn:support}
  \tfrac{1}{n}\innerprod{p_n}{r_n} = \support{\setD}\left( \tfrac{1}{n} r_n \right).
\end{equation}
Taking the limit of \eqref{eqn:support} and using \eqref{eqn:sup_lim_ineq}, we obtain
\[
  \lim_{n\to\infty}\tfrac{1}{n}\innerprod{p_n}{r_n} = \lim_{n\to\infty} \sup_{p\in\setD} \innerprod{p}{\tfrac{1}{n}r_n} \ge \sup_{p\in\setD}\lim_{n\to\infty}\innerprod{p}{\tfrac{1}{n}r_n} = \support{\setD}(\dlt r).
\]
Since $p_n\in\setD$, we also have
\[
  \lim_{n\to\infty}\tfrac{1}{n}\innerprod{p_n}{r_n} \le \sup_{p\in\setD} \lim_{n\to\infty}\innerprod{p}{\tfrac{1}{n}r_n} = \support{\setD}(\dlt r).
\]
The result follows by combining the two inequalities above.
\end{proof}

The results of \Prop~\ref{prop:dpdr} are straightforward under the additional assumption that $\setD$ is compact, since then $\recession{\setD}=\{0\}$ and $\polar{(\recession{\setD})}=\mcf{H}$, and thus
\begin{align*}
  \lim_{n\to\infty}\tfrac{1}{n} p_n &= \lim_{n\to\infty}\tfrac{1}{n}\project{\setD} s_n = 0 = \project{\recession{\setD}}(\dlt s) \\
  \lim_{n\to\infty}\tfrac{1}{n} r_n &= \lim_{n\to\infty}\tfrac{1}{n} (s_n-p_n) = \dlt s = \project{\polar{(\recession{\setD})}}(\dlt s).
\end{align*}
Moreover, the compactness of $\setD$ implies the continuity of $\support{\setD}$ \cite[\Ex~11.2]{Bauschke:2017:book}, and thus taking the limit of \eqref{eqn:support} yields
\[
  \lim_{n\to\infty}\tfrac{1}{n}\innerprod{p_n}{r_n} = \lim_{n\to\infty}\support{\setD}\left(\tfrac{1}{n}r_n\right) = \support{\setD}\left(\lim_{n\to\infty}\tfrac{1}{n}r_n\right) = \support{\setD}(\dlt r).
\]
When $\setD$ is a (translated) closed convex cone, its recession cone is the cone itself, and the results of \Prop~\ref{prop:dpdr} can be shown using Moreau's decomposition and some basic properties of the projection operator; see \cite[\Lem~A.3 \& \Lem~A.4]{Banjac:2019} for details.

A result that motivated our generalization of these limits to an arbitrary nonempty closed convex set $\setD$ is given in \cite[\Lem~4.3]{Shen:2015}, where \Prop~\ref{prop:dpdr}\ref{prop:dpdr:ii} is established in a finite-dimensional setting.

\section{Douglas-Rachford Algorithm}\label{sec:dra}
The Douglas-Rachford algorithm is an operator splitting method, which can be used to solve composite minimization problems of the form
\begin{equation}\label{eqn:composite_problem}
	\underset{w\in\mcf{H}}{\rm minimize} \quad f(w) + g(w),
\end{equation}
where $f$ and $g$ are proper lower semicontinuous convex functions.
An iteration of the algorithm in application to problem~\eqref{eqn:composite_problem} can be written as
\begin{align*}
  w_n			    &= \prox_g s_n \\
  \tilde{w}_n	&= \prox_f ( 2w_n - s_n ) \\
  s_{n+1}		  &= s_n + \alpha ( \tilde{w}_n - w_n ).
\end{align*}
where $\alpha\in\left] 0,2 \right[$ is the \emph{relaxation parameter}.

If we rewrite problem~\eqref{eqn:primal} as
\begin{align*}
  f(x,z) &= \half \innerprod{Qx}{x} + \innerprod{q}{x} + \indicator{Ax=z}(x,z) \\
  g(x,z) &= \indicator{\setC}(z),
\end{align*}
then an iteration of the Douglas-Rachford algorithm takes the following form \cite{Banjac:2019,Stellato:2020}:
\begin{subequations}\label{alg:osqp_avg}
\begin{align}
  \tilde{x}_n &= \argmin_{x\in\mcf{H}_1} \big( \half\innerprod{Qx}{x} + \innerprod{q}{x} + \half \norm{x - x_n}^2 + \half \norm{Ax - (2\project{\setC} - \Id) v_n}^2 \big) \label{eqn:osqp_eq_constr_qp} \\
	x_{n+1} &= x_n + \alpha \left( \tilde{x}_n - x_n \right) \label{eqn:osqp_drs_x_update} \\
	v_{n+1} &= v_n + \alpha \left( A\tilde{x}_n - \project{\setC} v_n \right) \label{eqn:osqp_drs_v_update}
\end{align}
\end{subequations}
We will exploit the following well-known result to analyze the asymptotic behavior of the algorithm \cite{Lions:1979}:
\begin{fact}\label{fact:dra}
	Iteration~\eqref{alg:osqp_avg} amounts to
	\[
		(x_{n+1},v_{n+1}) = T_{\rm DR}(x_n,v_n),
	\]
	where $T_{\rm DR}\colon(\mcf{H}_1\times\mcf{H}_2)\to(\mcf{H}_1\times\mcf{H}_2)$ is an $(\alpha/2)$-averaged operator.
\end{fact}

The solution to the subproblem in~\eqref{eqn:osqp_eq_constr_qp} satisfies the optimality condition
\begin{equation}\label{eqn:osqp_linsys}
  Q\tilde{x}_n + q + (\tilde{x}_n - x_n) + A^* \left( A\tilde{x}_n - (2\project{\setC} - \Id) v_n \right) = 0.
\end{equation}
If we rearrange~\eqref{eqn:osqp_drs_x_update} to isolate~$\tilde{x}_n$,
\[
  \tilde{x}_n = x_n + \alpha\inv \dlt x_{n+1},
\]
and substitute it into \eqref{eqn:osqp_drs_v_update} and \eqref{eqn:osqp_linsys}, we obtain the following relations between the iterates:
\begin{subequations}\label{eqn:residUpdates}
\begin{align}
	Ax_n - \project{\setC} v_n &= -\alpha\inv \left( A\dlt x_{n+1} - \dlt v_{n+1}\right) \label{eqn:residUpdate:primal} \\
	Qx_n + q + A^*(\Id - \project{\setC}) v_n &= -\alpha\inv \left( (Q+\Id)\dlt x_{n+1} + A^* \dlt v_{n+1} \right). \label{eqn:residUpdate:dual}
\end{align}
\end{subequations}
Let us define the following auxiliary iterates of iteration~\eqref{alg:osqp_avg}:
\begin{subequations}\label{eqn:zy_from_v}
  \begin{align}
    z_n &\eqdef \project{\setC} v_n \label{eqn:z_from_v} \\
    y_n &\eqdef (\Id - \project{\setC}) v_n. \label{eqn:y_from_v}
  \end{align}
\end{subequations}
Observe that the pair $(z_n,y_n)$ satisfies optimality condition \eqref{eqn:inclusion_cond} for all $n\in\Nat$ \cite[\Prop~6.47]{Bauschke:2017:book}, and that the right-hand terms in \eqref{eqn:residUpdates} indicate how far the iterates $(x_n,z_n,y_n)$ are from satisfying \eqref{eqn:prim_feas} and \eqref{eqn:dual_feas}.

The following corollary follows directly from Fact~\ref{fact:operatorConvergence}, \Prop~\ref{prop:dpdr}, Fact~\ref{fact:dra}, and Moreau's decomposition \cite[\Thm~6.30]{Bauschke:2017:book}:
\begin{corollary}\label{cor:dr}
Let the sequences $\seq{x_n}$, $\seq{v_n}$, $\seq{z_n}$, and $\seq{y_n}$ be given by \eqref{alg:osqp_avg} and \eqref{eqn:zy_from_v}, and $(\dlt x, \dlt v) \eqdef \project{\closure{\range}(T_{\rm DR}-\Id)}(0)$.
Then
\begin{enumerate}[label=(\roman*)]
  \item \label{cor:dr:xv} $\tfrac{1}{n} (x_n,v_n) \to (\dlt x,\dlt v)$.
  \item \label{cor:dr:dlt_xv} $(\dlt x_n, \dlt v_n) \to (\dlt x,\dlt v)$.
  \item \label{cor:dr:yn} $y_n \in \polar{(\recession{\setC})}$.
  \item \label{cor:dr:dz} $\tfrac{1}{n} z_n \to \dlt z \eqdef \project{\recession{\setC}}(\dlt v)$.
  \item \label{cor:dr:dy} $\tfrac{1}{n} y_n \to \dlt y \eqdef \project{\polar{(\recession{\setC})}}(\dlt v)$.
  \item \label{cor:dr:support} $\lim_{n\to\infty}\tfrac{1}{n}\innerprod{z_n}{y_n} = \support{\setC}(\dlt y)$.
  \item \label{cor:dr:moreau1} $\dlt z + \dlt y = \dlt v$.
  \item \label{cor:dr:moreau2} $\innerprod{\dlt z}{\dlt y} = 0$.
  \item \label{cor:dr:moreau3} $\norm{\dlt z}^2 + \norm{\dlt y}^2 = \norm{\dlt v}^2$.
\end{enumerate}
\end{corollary}

The following two propositions generalize \cite[\Prop~5.1 \& \Prop~5.2]{Banjac:2019}, though the proofs follow very similar arguments.

\begin{proposition}\label{prop:limits1}
The following relations hold between $\dlt x$, $\dlt z$, and~$\dlt y$, which are defined in \Cor~\ref{cor:dr}:
\begin{enumerate}[label=(\roman*)]
  \item \label{prop:limits1:Adx}  $A \dlt x = \dlt z$.
  \item \label{prop:limits1:Qdx}  $Q\dlt x = 0$.
  \item \label{prop:limits1:Atdy} $A^* \dlt y = 0$.
  \item \label{prop:limits1:dz}   $\dlt z_n \to \dlt z$.
  \item \label{prop:limits1:dy}   $\dlt y_n \to \dlt y$.
\end{enumerate}
\end{proposition}
\begin{proof}
\ref{prop:limits1:Adx}:
Divide~\eqref{eqn:residUpdate:primal} by $n$, take the limit, and use \Cor~\ref{cor:dr}\ref{cor:dr:dz} to get
\begin{equation}\label{eqn:Adx}
  A \dlt x = \lim_{n\to\infty}\tfrac{1}{n}\project{\setC} v_n = \dlt z.
\end{equation}

\ref{prop:limits1:Qdx}:
Divide~\eqref{eqn:residUpdate:dual} by $n$, take the inner product of both sides with $\dlt x$ and take the limit to obtain
\[
	\innerprod{Q\dlt x}{\dlt x} = -\lim_{n\to\infty} \big\langle A \dlt x, \tfrac{1}{n}(\Id-\project{\setC}) v_n \big\rangle = -\innerprod{\dlt z}{\dlt y} = 0,
\]
where we used \eqref{eqn:Adx} and \Cor~\ref{cor:dr}\ref{cor:dr:dy} in the second equality, and \Cor~\ref{cor:dr}\ref{cor:dr:moreau2} in the third.
Due to \cite[\Cor~18.18]{Bauschke:2017:book}, the equality above implies
\begin{equation}\label{eqn:Qdx}
	Q \dlt x = 0.
\end{equation}

\ref{prop:limits1:Atdy}:
Divide \eqref{eqn:residUpdate:dual} by $n$, take the limit, and use \eqref{eqn:Qdx} to obtain
\[
  0 = \lim_{n\to\infty}\tfrac{1}{n} A^* (\Id-\project{\setC}) v_n = A^* \dlt y,
\]
where we used \Cor~\ref{cor:dr}\ref{cor:dr:dy} in the second equality.

\ref{prop:limits1:dz}:
Subtracting \eqref{eqn:residUpdate:primal} at iterations $n+1$ and $n$, and taking the limit yield
\[
  \lim_{n\to\infty}\dlt z_n = A\dlt x = \dlt z,
\]
where the second equality follows from~\eqref{eqn:Adx}.

\ref{prop:limits1:dy}:
From \eqref{eqn:zy_from_v} we have
\begin{align*}
	\lim_{n\to\infty} \dlt y_n = \lim_{n\to\infty} \left( \dlt v_n - \dlt z_n \right) = \dlt v - \dlt z = \dlt y,
\end{align*}
where the last equality follows from \Cor~\ref{cor:dr}\ref{cor:dr:moreau1}.
\end{proof}

\begin{proposition}\label{prop:limits2}
The following identities hold for $\dlt x$ and $\dlt y$, which are defined in \Cor~\ref{cor:dr}:
\begin{enumerate}[label=(\roman*)]
  \item \label{prop:limits2:qdx}
  $\innerprod{q}{\dlt x} = -\alpha\inv \norm{\dlt x}^2 - \alpha\inv \norm{A\dlt x}^2$.

  \item \label{prop:limits2:support}
  $\support{\setC}(\dlt y) = -\alpha\inv \norm{\dlt y}^2$.
\end{enumerate}
\end{proposition}
\begin{proof}
Take the inner product of both sides of~\eqref{eqn:residUpdate:dual} with~$\dlt x$ and use \eqref{eqn:Qdx} to obtain
\[
	\innerprod{q}{\dlt x} + \innerprod{A\dlt x}{y_n} = - \alpha\inv \innerprod{\dlt x}{\dlt x_{n+1}} - \alpha\inv \innerprod{A\dlt x}{\dlt v_{n+1}}.
\]
Taking the limit and using \Prop~\ref{prop:limits1}\ref{prop:limits1:Adx} and \Cor~\ref{cor:dr}\ref{cor:dr:moreau1}\&\ref{cor:dr:moreau2} give
\begin{equation}\label{eqn:ineq1}
  \innerprod{q}{\dlt x} + \alpha\inv \norm{\dlt x}^2 + \alpha\inv \norm{\dlt z}^2 = -\lim_{n\to\infty} \innerprod{\dlt z}{y_n} \ge 0,
\end{equation}
where the inequality follows from \Cor~\ref{cor:dr}\ref{cor:dr:yn}\&\ref{cor:dr:dz} as the inner product of terms in $\recession{\setC}$ and $\polar{(\recession{\setC})}$ is nonpositive.
Now take the inner product of both sides of~\eqref{eqn:residUpdate:primal} with $\dlt y$ to obtain
\[
	\innerprod{A^* \dlt y}{x_n + \alpha\inv \dlt x_{n+1}} - \innerprod{\dlt y}{\project{\setC} v_n} = \alpha\inv \innerprod{\dlt y}{\dlt v_{n+1}}.
\]
Due to \Prop~\ref{prop:limits1}\ref{prop:limits1:Atdy}, the first inner product on the left-hand side is zero.
Taking the limit and using \Cor~\ref{cor:dr}\ref{cor:dr:moreau1}\&\ref{cor:dr:moreau2}, we obtain
\[
  -\alpha\inv \norm{\dlt y}^2 = \lim_{n\to\infty}\innerprod{\dlt y}{\project{\setC} v_n} \le \sup_{z\in\setC} \innerprod{\dlt y}{z} = \support{\setC}(\dlt y),
\]
or equivalently,
\begin{equation}\label{eqn:ineq2}
  \support{\setC}(\dlt y) + \alpha\inv \norm{\dlt y}^2 \ge 0.
\end{equation}
Summing \eqref{eqn:ineq1} and \eqref{eqn:ineq2} and using \Cor~\ref{cor:dr}\ref{cor:dr:moreau3}, we obtain
\begin{equation}\label{eqn:ineq3}
	\innerprod{q}{\dlt x} + \support{\setC}(\dlt y) + \alpha\inv \norm{\dlt x}^2 + \alpha\inv \norm{\dlt v}^2 \ge 0.
\end{equation}
Now take the inner product of both sides of~\eqref{eqn:residUpdate:dual} with $x_n$ to obtain
\[
	\innerprod{Qx_n}{x_n} + \innerprod{q}{x_n} + \innerprod{Ax_n}{y_n} = -\alpha\inv \innerprod{(Q+\Id)\dlt x_{n+1}}{x_n} - \alpha\inv \innerprod{Ax_n}{\dlt v_{n+1}}.
\]
Dividing by $n$, taking the limit, and using \Prop~\ref{prop:limits1}\ref{prop:limits1:Adx}\&\ref{prop:limits1:Qdx} and \Cor~\ref{cor:dr}\ref{cor:dr:moreau1}\&\ref{cor:dr:moreau2} yield
\[
	\lim_{n\to\infty} \tfrac{1}{n} \innerprod{Qx_n}{x_n} + \innerprod{q}{\dlt x} + \lim_{n\to\infty} \tfrac{1}{n}\innerprod{Ax_n}{y_n} = - \alpha\inv \norm{\dlt x}^2 - \alpha\inv \norm{\dlt z}^2.
\]
We can write the last term on the left-hand side as
\begin{align*}
  \lim_{n\to\infty} \tfrac{1}{n}\innerprod{Ax_n}{y_n} &= \lim_{n\to\infty} \tfrac{1}{n}\innerprod{z_n + \alpha\inv \left( \dlt v_{n+1} - A\dlt x_{n+1} \right)}{y_n} \\
  &= \lim_{n\to\infty} \tfrac{1}{n}\innerprod{z_n}{y_n} + \alpha\inv \norm{\dlt y}^2 \\
  &= \support{\setC}(\dlt y) + \alpha\inv \norm{\dlt y}^2,
\end{align*}
where the first equality follows from \eqref{eqn:residUpdate:primal}, the second from \Prop~\ref{prop:limits1}\ref{prop:limits1:Adx} and \Cor~\ref{cor:dr}\ref{cor:dr:dy}\&\ref{cor:dr:moreau1}, and the third from \Cor~\ref{cor:dr}\ref{cor:dr:support}.
Plugging the equality above in the preceding, we obtain
\begin{equation}\label{eqn:ineq4}
	\innerprod{q}{\dlt x} + \support{\setC}(\dlt y) + \alpha\inv \norm{\dlt x}^2 + \alpha\inv \norm{\dlt v}^2 = -\lim_{n\to\infty} \tfrac{1}{n} \innerprod{Qx_n}{x_n} \le 0,
\end{equation}
where the inequality follows from the monotonicity of $Q$.
Comparing inequalities \eqref{eqn:ineq3} and \eqref{eqn:ineq4}, it follows that they must be satisfied with equality.
Consequently, the left-hand sides of \eqref{eqn:ineq1} and \eqref{eqn:ineq2} must be zero.
This concludes the proof.
\end{proof}

Given the infeasibility conditions in \Prop~\ref{prop:infeas}, it follows from \Prop~\ref{prop:limits1} and \Prop~\ref{prop:limits2} that, if the limit $\dlt y$ is nonzero, then problem~\eqref{eqn:primal} is strongly infeasible, and similarly, if $\dlt x$ is nonzero, then its dual is strongly infeasible.
Thanks to the fact that $(\dlt y_n, \dlt x_n) \to (\dlt y, \dlt x)$, we can now extend the termination criteria proposed in \cite[\S 5.2]{Banjac:2019} for the more general case where $\setC$ is a general nonempty closed convex set.
The criteria in \cite[\S 5.2]{Banjac:2019} evaluate conditions given in \Prop~\ref{prop:infeas} at $\dlt y_n$ and $\dlt x_n$, and have already formed the basis for stable numerical implementations \cite{Stellato:2020,Garstka:2019}.
Our results pave the way for similar developments in the more general setting considered here.

\section{Proximal-Point Algorithm}\label{sec:ppa}
The proximal-point algorithm is a method for finding a vector $w\in\mcf{H}$ that solves the following inclusion problem:
\begin{equation}\label{eqn:monotone_inclusion}
  0 \in B(w),
\end{equation}
where $B\colon\mcf{H}\to2^\mcf{H}$ is a maximally monotone operator.
An iteration of the algorithm in application to problem~\eqref{eqn:monotone_inclusion} can be written as
\[
  w_{n+1} = (\Id + \stepsize B)^{-1} w_n,
\]
where $\stepsize>0$ is the \emph{regularization parameter}.

Due to \cite[\Cor~16.30]{Bauschke:2017:book}, we can rewrite \eqref{eqn:opt_incl} as
\[
  0 \in \mcf{M}(x,y) \eqdef \begin{pmatrix} Qx + q + A^* y \\ -Ax + \partial \indicator{\setC}^* (y) \end{pmatrix},
\]
where $\mcf{M}\colon(\mcf{H}_1\times\mcf{H}_2)\to2^{(\mcf{H}_1\times\mcf{H}_2)}$ is a maximally monotone operator \cite{Rockafellar:1976}.
An iteration of the proximal-point algorithm in application to the inclusion above is then
\begin{equation}\label{alg:ppa}
  (x_{n+1},y_{n+1}) = \left( \Id + \stepsize \mcf{M} \right)\inv (x_n,y_n),
\end{equation}
which was also analyzed in \cite{Hermans:2019}.
We will exploit the following result \cite[\Prop~23.8]{Bauschke:2017:book} to analyze the algorithm:
\begin{fact}\label{fact:ppa}
	Operator $T_{\rm PP} \eqdef (\Id + \stepsize \mcf{M})\inv$ is the resolvent of a maximally monotone operator and is thus $(1/2)$-averaged.
\end{fact}

Iteration \eqref{alg:ppa} reads
\begin{subequations}\label{eqn:ppa:xy_next}
\begin{align}
  0 &=   x_{n+1} - x_n + \stepsize \left( Q x_{n+1} + q  + A^* y_{n+1} \right) \label{eqn:ppa:x_next} \\
  0 &\in y_{n+1} - y_n + \stepsize \left( -A x_{n+1} + \partial \indicator{\setC}^* (y_{n+1}) \right). \label{eqn:ppa:y_next}
\end{align}
\end{subequations}
Inclusion \eqref{eqn:ppa:y_next} can be written as
\[
  \stepsize A x_{n+1} + y_n \in \left( \Id + \stepsize \partial \indicator{\setC}^* \right) y_{n+1},
\]
which is equivalent to \cite[\Prop~16.44]{Bauschke:2017:book}
\begin{equation}\label{eqn:ppa:y_next_proj}
  y_{n+1} = \prox_{\stepsize \indicator{\setC}^*} \left( \stepsize A x_{n+1} + y_n \right)
  = \stepsize A x_{n+1} + y_n - \stepsize \project{\setC} ( A x_{n+1} + \stepsize\inv y_n),
\end{equation}
where the second equality follows from \cite[\Thm~14.3]{Bauschke:2017:book}.
Let us define the following auxiliary iterates of iteration~\eqref{alg:ppa}:
\begin{subequations}\label{eqn:ppa:vz}
\begin{align}
  v_{n+1} &\eqdef Ax_{n+1} + \stepsize\inv y_n \\
  z_{n+1} &\eqdef \project{\setC} v_{n+1},
\end{align}
\end{subequations}
and observe from \eqref{eqn:ppa:y_next_proj} that
\[
  y_{n+1} = \stepsize (\Id-\project{\setC}) v_{n+1}.
\]
Using \eqref{eqn:ppa:x_next} and \eqref{eqn:ppa:y_next_proj}, we now obtain the following relations between the iterates:
\begin{subequations}\label{eqn:ppa:residUpdates}
\begin{align}
  A x_{n+1} - \project{\setC} v_{n+1} &= \stepsize\inv \dlt y_{n+1} \label{eqn:ppa:residUpdate:primal} \\
  Q x_{n+1} + q  + \stepsize A^* (\Id-\project{\setC}) v_{n+1} &= -\stepsize\inv \dlt x_{n+1}. \label{eqn:ppa:residUpdate:dual}
\end{align}
\end{subequations}
Similarly as for the Douglas-Rachford algorithm, the pair $(z_{n+1},y_{n+1})$ satisfies optimality condition \eqref{eqn:inclusion_cond} for all $n\in\Nat$.
Observe that the optimality residuals, given by the norms of the left-hand terms in \eqref{eqn:ppa:residUpdates}, can be computed by evaluating the norms of $\dlt y_{n+1}$ and $\dlt x_{n+1}$.

The following corollary follows directly from Fact~\ref{fact:operatorConvergence}, \Prop~\ref{prop:dpdr}, and Fact~\ref{fact:ppa}:
\begin{corollary}\label{cor:pp}
Let the sequences $\seq{x_n}$, $\seq{y_n}$, $\seq{v_n}$, and $\seq{z_n}$ be given by \eqref{alg:ppa} and \eqref{eqn:ppa:vz}, and $(\dlt x, \dlt y) \eqdef \project{\closure{\range}(T_{\rm PP}-\Id)}(0)$.
Then
\begin{enumerate}[label=(\roman*)]
  \item \label{cor:pp:xyv} $\tfrac{1}{n} (x_n,y_n,v_n) \to (\dlt x,\dlt y,A\dlt x + \stepsize\inv\dlt y)$.
  \item \label{cor:pp:dlt_xyv} $(\dlt x_n,\dlt y_n,\dlt v_n) \to (\dlt x,\dlt y,A\dlt x + \stepsize\inv\dlt y)$.
  \item \label{cor:pp:yn} $y_{n+1}\in\polar{(\recession{\setC})}$.
  \item \label{cor:pp:dz} $\tfrac{1}{n}z_n\to\dlt z\eqdef\project{\recession{\setC}}(\dlt v)$.
  \item \label{cor:pp:dy} $\dlt y=\stepsize\project{\polar{(\recession{\setC})}}(\dlt v)$.
  \item \label{cor:pp:support} $\lim_{n\to\infty}\tfrac{1}{n}\innerprod{z_n}{y_n} = \support{\setC}(\dlt y)$.
\end{enumerate}
\end{corollary}

The proofs of the following two propositions follow similar arguments as those in Section~\ref{sec:dra}, and are thus omitted.

\begin{proposition}\label{prop:ppa:limits1}
The following relations hold between $\dlt x$, $\dlt z$, and~$\dlt y$, which are defined in \Cor~\ref{cor:pp}:
\begin{enumerate}[label=(\roman*)]
  \item \label{prop:ppa:limits1:Adx}  $A \dlt x = \dlt z$.
  \item \label{prop:ppa:limits1:Qdx}  $Q\dlt x = 0$.
  \item \label{prop:ppa:limits1:Atdy} $A^* \dlt y = 0$.
\end{enumerate}
\end{proposition}

\begin{proposition}\label{prop:ppa:limits2}
  The following identities hold for $\dlt x$ and $\dlt y$, which are defined in \Cor~\ref{cor:pp}:
\begin{enumerate}[label=(\roman*)]
  \item \label{prop:ppa:limits2:qdx}
  $\innerprod{q}{\dlt x} = -\stepsize\inv \norm{\dlt x}^2$.

  \item \label{prop:ppa:limits2:support}
  $\support{\setC}(\dlt y) = -\stepsize\inv \norm{\dlt y}^2$.
\end{enumerate}
\end{proposition}

The authors in \cite{Hermans:2019} use similar termination criteria to those given in \cite[\S 5.2]{Banjac:2019} to detect infeasibility of convex quadratic programs using the algorithm given by iteration \eqref{alg:ppa}, though they do not prove that $\dlt y$ and $\dlt x$ are indeed infeasibility certificates whenever the problem is strongly infeasible.
Identities in \eqref{eqn:ppa:residUpdates} show that, when $(\dlt y,\dlt x)=(0,0)$, the optimality conditions \eqref{eqn:kkt} are satisfied in the limit.
Otherwise, \Prop~\ref{prop:infeas}, \Prop~\ref{prop:ppa:limits1}, and \Prop~\ref{prop:ppa:limits2} imply that problem~\eqref{eqn:primal} and/or its dual is strongly infeasible.

\begin{remark}
  Weak infeasibility of problem~\eqref{eqn:primal} means that the sets $\range{A}$ and $\setC$ do not intersect, but the distance between them is zero.
  In such cases, there exists no $\bar{y}\in\mcf{H}_2$ satisfying the conditions in \Prop~\ref{prop:infeas} and the algorithms studied in Sections~\ref{sec:dra}--\ref{sec:ppa} would yield $\dlt y_n \to \dlt y = 0$.
  A similar reasoning holds for the weak infeasibility of the dual problem for which the algorithms would yield $\dlt x_n \to \dlt x = 0$.
\end{remark}

\section*{Acknowledgements}
This project has received funding from the European Research Council (ERC) under the European Union's Horizon 2020 research and innovation programme grant agreement OCAL, No.\ 787845.

\bibliography{refs}

\end{document}